\documentclass[11pt]{article}
\usepackage{amsfonts}
\usepackage{amssymb}
\usepackage{amsthm}
\usepackage{amsmath}

\usepackage{bm}
\usepackage{amscd}
\usepackage[frenchb]{babel}
\usepackage[latin1]{inputenc}
\usepackage[T1]{fontenc}

\numberwithin{equation}{section} \DeclareMathSizes{2}{10}{12}{13}
\newtheorem{thm}{Proposition}[section]
\newtheorem{Thm}[thm]{Th\'{e}or\`{e}me}

\newtheorem{lem}[thm]{Lemme}
\newtheorem{defn}[thm]{D\'{e}finition}

\title{Sur la cat\'{e}gorie d\'{e}riv\'{e}e des faisceaux tordus}
\author{Abhishek Banerjee}
\date{}

\begin{document}

\maketitle

\medskip

\medskip
\centerline{\emph{Dept. of Mathematics, Indian Institute of Science, Bengaluru - 560012, Karnataka, India}}

\centerline{ \emph{Email: abhishekbanerjee1313@gmail.com}}

\medskip

\medskip
\begin{abstract} 

Soit $X$ un sch\'{e}ma quasi-compact et s\'{e}par\'{e} et soit $\alpha\in \check{C}^2(X,
\mathcal O_X^*)$ un cocycle de C\v{e}ch. Nous consid\'{e}rons la cat\'{e}gorie d\'{e}riv\'{e}e 
$D(QCoh(X,\alpha))$ des faisceaux quasi-coh\'{e}rents sur $X$ tordu par $\alpha$. Soit $\Delta(X,\alpha)$ la
plus petite sous-cat\'{e}gorie triangul\'{e}e de $D(QCoh(X,\alpha))$ contenant tous les objets
$u_*\mathcal F^\bullet$, o\`{u} $u:U\longrightarrow X$ est une immersion ouverte avec $U$ affine et
$\mathcal F^\bullet\in D(QCoh(U,u^*(\alpha)))$. Alors, le but de cet article est de montrer que $\Delta(X,\alpha)=D(QCoh(X,\alpha))$. 

\medskip

\centerline{\bf Abstract}

\medskip
Let $X$ be a quasi-compact and separated scheme and let $\alpha\in \check{C}^2(X,
\mathcal O_X^*)$ be a C\v{e}ch cocycle. We consider the derived category $D(QCoh(X,\alpha))$ of 
quasi-coherent sheaves on $X$ twisted by $\alpha$. Let $\Delta(X,\alpha)$ be the smallest
triangulated subcategory of $D(QCoh(X,\alpha))$ containing all the objects $u_*\mathcal F^\bullet$, where $u:U\longrightarrow X$ is an open immersion with $U$ affine and $\mathcal F^\bullet\in D(QCoh(U,u^*(\alpha)))$. Then, the purpose of this article is to show that $\Delta(X,\alpha)=D(QCoh(X,\alpha))$. 
\end{abstract}

\medskip
\noindent \emph{MSC(2010) Subject Classification:  \bf 14F05}


\medskip

\medskip
\section{Introduction}

\medskip

\medskip
Soit $X$ un sch\'{e}ma quasi-compact et s\'{e}par\'{e} et soit $\alpha\in \check{C}^2(X,\mathcal O_X^*)$ un cocycle de C\v{e}ch. Alors, les faisceaux sur $X$ tordu par $\alpha$ ont \'{e}t\'{e} introduits d'abord
par Giraud \cite{Giraud} dans le cadre de ses \'{e}tudes sur la cohomologie non-ab\'{e}lienne. Quand $\alpha$ est trivial, les faisceaux tordus par $\alpha$ sont les m\^{e}mes que les faisceaux sur $X$. Dans \cite{AC1}, \cite{Cald2}, C\u ald\u araru a commenc\'{e} une \'{e}tude syst\'{e}matique de cette th\'{e}orie, en lien avec les faisceaux tordus sur les vari\'{e}t\'{e}s de Calabi-Yau. Plus r\'{e}cemment, la cat\'{e}gorie des faisceaux sur $X$ tordu par $\alpha$ 
a \'{e}t\'{e} \'{e}tudi\'{e}e par des autres auteurs (voir Antieau \cite{Anti},  Lieblich \cite{Max2}, \cite{Max}). Beaucoup de ces articles sont concern\'{e}s 
par la cat\'{e}gorie d\'{e}riv\'{e}e des faisceaux tordus. De plus, il existe des liens intimes entre les faisceaux tordus et les groupes de Brauer ainsi
que les alg\`{e}bres d'Azumaya (voir, par exemple, \cite[Chapter 4]{Huy}). Dans cet article, nous consid\'{e}rons la cat\'{e}gorie d\'{e}riv\'{e}e 
$D(QCoh(X,\alpha))$ des faisceaux quasi-coh\'{e}rents sur $X$ tordu par $\alpha$. Nous \'{e}tablissons le th\'{e}or\`{e}me suivant sur la cat\'{e}gorie d\'{e}riv\'{e}e $D(QCoh(X,\alpha))$. 

\medskip

\begin{Thm}\label{Th1} Soit $X$ un sch\'{e}ma quasi-compact et s\'{e}par\'{e} et soit $\alpha\in \check{C}^2(X,
\mathcal O_X^*)$ un cocycle de C\v{e}ch. Soit $\Delta(X,\alpha)$ la
plus petite sous-cat\'{e}gorie triangul\'{e}e de $D(QCoh(X,\alpha))$ contenant tous les objets
$u_*\mathcal F^\bullet$, o\`{u} $u:U\longrightarrow X$ est une immersion ouverte avec $U$ affine et
$\mathcal F^\bullet\in D(QCoh(U,u^*(\alpha)))$. Alors,  $\Delta(X,\alpha)=D(QCoh(X,\alpha))$. 
\end{Thm}

En particulier, supposons que $\alpha$ est trivial. Alors $D(QCoh(X,\alpha))=D(QCoh(X))$, la cat\'{e}gorie d\'{e}riv\'{e}e des faisceaux quasi-coh\'{e}rents sur $X$.  Il s'ensuit de Theorem \ref{Th1}  que $D(QCoh(X))=\Delta(X)$.  Ici, $\Delta(X)$ est la plus petite sous-cat\'{e}gorie
triangul\'{e}e de $D(QCoh(X))$ contenant tous les objets $u_*\mathcal F^\bullet$ avec $u:U
\longrightarrow X$ une immersion affine et $\mathcal F^\bullet\in D(QCoh(U))$. Nous notons
que le sch\'{e}ma $X$ n'est pas n\'{e}cessairement noeth\'{e}rien. Pour les sch\'{e}mas noeth\'{e}riens, la th\'{e}orie de la cat\'{e}gorie d\'{e}riv\'{e}e des faisceaux quasi-coh\'{e}rents est classique (voir, par exemple, Hartshorne \cite{Hart} pour une exposition). Mais, lorsque $X$ n'est pas noeth\'{e}rien, la
situation est un peu plus compliqu\'{e}e (voir, par exemple, les travaux de Lipman \cite{Lip1},\cite{Lip2} et Neeman \cite{Neeman}). Plus g\'{e}n\'{e}ralement, il est naturel de se demander
si on peut avoir une th\'{e}orie similaire pour la cat\'{e}gorie d\'{e}riv\'{e}e des faisceaux quasi-coh\'{e}rents sur $X$ tordus par $\alpha$. Alors, le r\'{e}sultat de Th\'{e}or\`{e}me \ref{Th1} nous permet
d'\'{e}tudier la cat\'{e}gorie $D(QCoh(X,\alpha))$  en termes des cat\'{e}gories $D(QCoh(U,u^*(\alpha)))$ avec $U$ affine. 

\medskip
L'organisation de cet article est comme suit: nous commen\c{c}ons en rappelant les d\'{e}finitions
et quelques propri\'{e}t\'{e}s des faisceaux quasi-coh\'{e}rents sur $X$ tordus par $\alpha$ (voir \cite{AC1}). Si $f:Y\longrightarrow X$ est un morphisme des sch\'{e}mas et $\alpha
\in \check{C}^2(X,\mathcal O_X^*)$ est un cocycle de C\v{e}ch, nous consid\'{e}rons le foncteur
d'image directe: 
\begin{equation}
f_*:QCoh(Y,f^*(\alpha))\longrightarrow QCoh(X,\alpha)
\end{equation} On sait que $D(QCoh(X,\alpha))$ est une cat\'{e}gorie ab\'{e}lienne de Grothendieck
 (voir Antieau \cite{Anti}) et donc chaque objet dans $D(Qcoh(X,\alpha))$ dispose d'une r\'{e}solution
 $K$-injective (voir Section 2). De m\^{e}me, chaque objet de $D(QCoh(Y,f^*(\alpha)))$ dispose 
 d'une r\'{e}solution $K$-injective. Il s'ensuit qu'on a un foncteur d\'{e}riv\'{e}:
 \begin{equation}
 R_{qc}f_*:D(QCoh(Y,f^*(\alpha)))\longrightarrow D(QCoh(X,\alpha))
 \end{equation} De plus, pour une immersion ouverte $u:U\longrightarrow X$ avec $U$ affine, nous
 montrons que le foncteur d'image directe $u_*:QCoh(U,u^*(\alpha))\longrightarrow QCoh(X,\alpha)$ est 
 exact. Alors, on a un foncteur $u_*:D(QCoh(U,u^*(\alpha)))\longrightarrow D(QCoh(X,\alpha))$ au niveau des cat\'{e}gories d\'{e}riv\'{e}es. Nous consid\'{e}rons maintenant  la plus petite 
 sous-cat\'{e}gorie triangul\'{e}e $\Delta(X,\alpha)$ de $D(QCoh(X,\alpha))$ contenant tous les objets
$u_*\mathcal F^\bullet$, o\`{u} $u:U\longrightarrow X$ est une immersion ouverte avec $U$ affine et
$\mathcal F^\bullet\in D(QCoh(U,u^*(\alpha)))$. Alors, on montre le r\'{e}sultat de Th\'{e}or\`{e}me \ref{Th1} par induction sur $n(X)$, le nombre minimal des affines \`{a} s\'{e}lectionner pour former un
recouvrement de $X$.

\medskip

\medskip

\section{La cat\'{e}gorie d\'{e}riv\'{e}e des faisceaux tordus}

\medskip

\medskip
Soit $k$ un anneau commutatif et soit $(X,\mathcal O_X)$ un sch\'{e}ma quasi-compact et s\'{e}par\'{e} sur 
$Spec(k)$. Nous notons par $ZarAff(X)$ la cat\'{e}gorie des immersions ouvertes
$j:U\longrightarrow X$ avec $U$ affine. Soit $\alpha\in  \check{C}^2(X,\mathcal O_X^*)$ un cocycle de C\v{e}ch. Si $\mathcal U=\{U_i\}_{i\in I}$ est un recouvrement de $X$ tel qu'on peut 
pr\'{e}senter $\alpha$ comme une famille $\alpha=\{\alpha_{ijk}\in \Gamma(U_i\cap U_j\cap U_k,\mathcal O_X^*)\}_{i,j,k\in I}$, on dit que $\alpha\in \check{C}^2(X,\mathcal O_X^*,\mathcal U)$. Rappelons
qu'un faisceau tordu par $\alpha$ est d\'{e}fini comme suit (voir, par exemple, \cite[Definition 1.2.1]{AC1}).

\medskip 
\begin{defn} Soit $\mathcal U=\{U_i\}_{i\in I}$ un recouvrement ouvert de $X$. 
Soit $\alpha\in \check{C}^2(X,\mathcal O_X^*)$ un cocycle de C\v{e}ch correspondant \`{a} 
une famille $\{\alpha_{ijk}\in \Gamma(U_i\cap U_j\cap U_k,\mathcal O_X^*)\}_{i,j,k\in I}$. Alors, un
faisceau $\mathcal F$ sur $X$ tordu par $\alpha$ est un couple $(\{\mathcal F_i\}_{i\in I},
\{\varphi_{ij}\}_{i,j\in I})$ v\'{e}rifiant les conditions suivantes:

\medskip
(1) Pour chaque $i\in I$, $\mathcal F_i$ est un faisceau des $\mathcal O_{U_i}$-modules
sur $U_i$. 

\medskip
(2) Pour $i,j\in I$, $\varphi_{ij}:\mathcal F_j|U_i\cap U_j\longrightarrow \mathcal F_i|U_i\cap U_j$
est un isomorphisme tel que:

\medskip
\hspace{0.1in} (a) Pour chaque $i\in I$, on a $\varphi_{ii}=1$. 

\medskip
\hspace{0.1in} (b) Pour $i,j\in I$, on a $\varphi_{ij}=\varphi_{ji}^{-1}$. 

\medskip
\hspace{0.1in} (c) Pour $i,j,k\in I$, l'isomorphisme $\varphi_{ij}\circ \varphi_{jk}\circ 
\varphi_{ik}^{-1}:\mathcal F_i|U_i\cap U_j\cap U_k\longrightarrow \mathcal 
F_i|U_i\cap U_j\cap U_k$ correspond \`{a} la 
multiplication $\alpha_{ijk}\bullet \_:
\mathcal F_i|U_i\cap U_j\cap U_k\longrightarrow \mathcal 
F_i|U_i\cap U_j\cap U_k$  sur $\mathcal F_i|U_i\cap U_j\cap U_k$. 

\medskip
De plus, on dit que $\mathcal F$ est quasi-coh\'{e}rent si, pour tout $i\in I$, 
$\mathcal F_i$ est un faisceau quasi-coh\'{e}rent sur $U_i$. Nous notons par
$QCoh(X,\alpha,\mathcal U)$ la cat\'{e}gorie des faisceaux quasi-coh\'{e}rents
tordus par $\alpha$. Si $\alpha$ est trivial, nous notons par $QCoh(X)$ la cat\'{e}gorie des
faisceaux quasi-coh\'{e}rents sur $X$. 
\end{defn}

\medskip
On sait que la cat\'{e}gorie $QCoh(X,\alpha,\mathcal U)$ ne d\'{e}pend pas du
choix du recouvrement $\mathcal U$. Plus pr\'{e}cis\'{e}ment, si on peut pr\'{e}senter ce cocycle $\alpha\in \check{C}^2(X,\mathcal O_X^*)$ comme une autre famille $\alpha=\{\alpha_{ijk}'\in \Gamma(U'_{i'}\cap U'_{j'}\cap U'_{k'},\mathcal O_X^*)\}_{i',j',k'\in I'}$ vis-\`{a}-vis d'un autre recouvrement $\mathcal U'=\{U_{i'}'\}_{i'\in I'}$ de $X$, on a une \'{e}quivalence $QCoh(X,\alpha,\mathcal U)\cong 
QCoh(X,\alpha,\mathcal U')$ des cat\'{e}gories (voir \cite[$\S$ 1]{AC1}). Alors, on va noter par $QCoh(X,\alpha)$ la cat\'{e}gorie
des faisceaux quasi-coh\'{e}rents tordus par $\alpha$. De plus, pour cocycles $\alpha$, $\alpha'
\in \check{C}^2(X,\mathcal O_X^*)$ tels que $\alpha=\alpha'$ dans $\check{H}^2(X,\mathcal O_X^*)$, les cat\'{e}gories $QCoh(X,\alpha)$ et $QCoh(X,\alpha')$ sont \'{e}quivalentes. 

\medskip
Nous rappelons maintenant quelques propri\'{e}t\'{e}s des faisceaux tordus sur un sch\'{e}ma $X$. Pour en savoir plus sur ces r\'{e}sultats, voir, par exemple, \cite[$\S$ 1]{AC1}. 

\medskip
(P1) Soit $\alpha\in \check{C}^2(X,\mathcal O_X^*)$ un cocycle de C\v{e}ch et soient $\mathcal U=\{U_i\}_{i\in I}$,
$\mathcal U'=\{U'_{i'}\}_{i'\in I'}$ deux recouvrements de $X$. Supposons  qu'on peut pr\'{e}senter ce
cocycle $\alpha$ comme une famille $\alpha=\{\alpha_{ijk}\in \Gamma(U_{i}\cap U_{j}\cap U_{k},\mathcal O_X^*)\}_{i,j,k\in I}$ vis-\`{a}-vis du $\mathcal U$ et comme $\alpha=\{\alpha'_{i'j'k'}\in \Gamma(U'_{i'}\cap U'_{j'}\cap U'_{k'},\mathcal O_X^*)\}_{i',j',k'\in I'}$ vis-\`{a}-vis du $\mathcal U'$ . Alors, si $\mathcal F=(\{\mathcal F_i\}_{i\in I},
\{\varphi_{ij}\}_{i,j\in I})$ est un objet de $QCoh(X,\alpha,\mathcal U)$, on peut pr\'{e}senter
$\mathcal F$ comme un objet de $QCoh(X,\alpha,\mathcal U')$. Autrement dit, il existe un couple
$(\{\mathcal F'_{i'}\}_{i'\in I'},
\{\varphi_{i'j'}\}_{i',j'\in I'})\in QCoh(X,\alpha,\mathcal U')$ corr\'{e}spondant \`{a} $\mathcal F$. 

\medskip
(P2) Soit $f:X\longrightarrow Y$ un morphisme des sch\'{e}mas et soit $\beta\in \check{C}^2(Y,
\mathcal O_Y^*)$ un cocycle de C\v{e}ch. Alors, si $\mathcal G$ est un faisceau quasi-coh\'{e}rent sur $Y$
tordu par $\beta$, $f^*\mathcal G$ est un faisceau quasi-coh\'{e}rent sur $X$ tordu par
$f^*(\beta)\in \check{C}^2(X,\mathcal O_X^*)$. 

\medskip
(P3) Soit $f:X\longrightarrow Y$ un morphisme quasi-compact et s\'{e}par\'{e} des sch\'{e}mas et
soit $\beta\in \check{C}^2(Y,
\mathcal O_Y^*)$ un cocycle de C\v{e}ch. Alors, si $\mathcal F$ est un faisceau quasi-coh\'{e}rent sur $X$ tordu par $f^*(\beta)\in \check{C}^2(X,\mathcal O_X^*)$, $f_*\mathcal F$ est un faisceau
quasi-coh\'{e}rent sur $Y$ tordu par $\beta$. 

\medskip
(P4) Soit $f:X\longrightarrow Y$ un morphisme quasi-compact et s\'{e}par\'{e} des sch\'{e}mas et
soit $\beta\in \check{C}^2(Y,
\mathcal O_Y^*)$ un cocycle de C\v{e}ch.  Alors, $(f^*,f_*)$ est un couple des foncteurs adjoints
entre les cat\'{e}gories $QCoh(Y,\beta)$ et $QCoh(X,f^*(\beta))$. 

\medskip

\begin{lem}\label{lem1}  Soient $f:Y\longrightarrow X$, $f':Y'\longrightarrow X'$ morphismes s\'{e}par\'{e}s et
quasi-compacts des sch\'{e}mas. Soit $\alpha\in \check{C}^2(X,\mathcal O_X^*)$ un cocycle de C\v{e}ch 
et soit $\mathcal F$ un faisceau quasi-coh\'{e}rent sur $Y$ tordu par $f^*(\alpha)$. \'{E}tant donn\'{e}
un carr\'{e} cart\'{e}sien des sch\'{e}mas
\begin{equation}
\begin{CD}
Y' @>f'>> X' \\
@VhVV @VgVV \\
Y  @>f>> X \\
\end{CD}
\end{equation} tel que $g$, $h$ sont des morphismes plats, on a un isomorphisme $g^*f_*\mathcal
F\cong f'_*h^*\mathcal F$ dans $QCoh(X',g^*(\alpha))$. 
\end{lem}

\begin{proof} Prenons un recouvrement $\mathcal U=\{U_i\}_{i\in I}$ de $X$ tel qu'on peut
pr\'{e}senter $\alpha$ comme une famille $\alpha=\{\alpha_{ijk}\in \Gamma(U_i\cap U_j\cap U_k,
\mathcal O_X^*)\}_{i,j,k\in I}$.  Posons:
\begin{equation}
V_i:=Y\times_XU_i \qquad V'_i:=Y'\times_XU_i \qquad U'_i:=X'\times_XU_i \qquad \forall\textrm{ }
i\in I
\end{equation} Alors, on peut pr\'{e}senter $f^*(\alpha)$ comme une famille $f^*(\alpha)=\{f^*(\alpha_{ijk})\in \Gamma(V_i\cap V_j\cap V_k,
\mathcal O_Y^*)\}_{i,j,k\in I}$ vis-\`{a}-vis du recouvrement $\mathcal V=\{V_i\}_{i\in I}$ de $Y$. Par d\'{e}finition, $\mathcal F$ est un objet de $QCoh(Y,f^*(\alpha))$. Donc, on peut \'{e}crire 
$\mathcal F$ comme un objet $\mathcal F=(\{\mathcal F_i\}_{i\in I},
\{\varphi_{ij}\}_{i,j\in I})$ de $QCoh(Y,f^*(\alpha),\mathcal V)$. De plus, pour $i,j\in I$, on a des carr\'{e}s
cart\'{e}siens:
\begin{equation}
\begin{array}{cc}
\begin{CD}
V_i' @>f'_i>> U'_i \\
@Vh_iVV @Vg_iVV \\
V_i @>f_i>> U_i \\
\end{CD} \quad \qquad & 
\begin{CD}
V_i'\cap V'_j @>f'_{ij}>> U'_i\cap U'_j \\
@Vh_{ij}VV @Vg_{ij}VV \\
V_i\cap V_j @>f_{ij}>> U_{i}\cap U_j \\
\end{CD}
\end{array}
\end{equation} Il est clair que $f_i$, $f'_i$ sont des morphismes quasi-compacts et s\'{e}par\'{e}s et
$g_i$, $h_i$ sont des morphismes plats. Il s'ensuit que, pour chaque $i\in I$, on a un isomorphisme naturel des foncteurs:
\begin{equation}\label{2.4x}
g_i^*f_{i*}\cong f'_{i*}h_i^*:QCoh(V_i)\longrightarrow QCoh(U'_i)
\end{equation} On voit que le cocycle $g^*(\alpha)$ est un \'{e}l\'{e}ment de $\check{C}^2(X',\mathcal O_{X'}^*,\{U_i'\}_{i\in I})$. Nous consid\'{e}rons maintenant les couples $g^*f_{*}\mathcal F=(\{g_i^*f_{i*}\mathcal F_i\}_{i\in I},
\{g_{ij}^*f_{ij*}\varphi_{ij}\}_{i,j\in I})$ et $f'_*h^*\mathcal F=(\{f'_{i*}h_i^*\mathcal F_i\}_{i\in I},
\{f'_{ij*}h_{ij}^*\varphi_{ij}\}_{i,j\in I})$. Alors,  il r\'{e}sulte de \eqref{2.4x} qu'on a un isomorphisme 
\begin{equation}
g^*f_{*}\mathcal F=(\{g_i^*f_{i*}\mathcal F_i\}_{i\in I},
\{g_{ij}^*f_{ij*}\varphi_{ij}\}_{i,j\in I})\cong (\{f'_{i*}h_i^*\mathcal F_i\}_{i\in I},
\{f'_{ij*}h_{ij}^*\varphi_{ij}\}_{i,j\in I})=f'_*h^*\mathcal F
\end{equation} dans la cat\'{e}gorie $QCoh(X',g^*(\alpha))\cong QCoh(X',g^*(\alpha),\{U'_i\}_{i\in I})$. 

\end{proof}

\medskip
\begin{lem}\label{lem2} Soit $X$ un sch\'{e}ma quasi-compact et s\'{e}par\'{e}. Soit $u:U\longrightarrow X$ une immersion ouverte avec $U$ affine et soit
$\alpha\in \check{C}^2(X,\mathcal O_X^*)$ un cocycle de C\v{e}ch. Alors, le foncteur
$u_*:QCoh(U,u^*(\alpha))\longrightarrow QCoh(X,\alpha)$ est exact. 
\end{lem}

\begin{proof} Puisque $X$ est s\'{e}par\'{e} et $U$ est affine, $u:U\longrightarrow X$ est un morphisme affine et donc un morphisme quasi-compact. De plus, le morphisme $u:U\longrightarrow X$ \'{e}tant une immersion ouverte, $u$ est s\'{e}par\'{e}. Alors, en cons\'{e}quence du fait
que $(u^*,u_*)$ est un couple des foncteurs adjoints, il s'ensuit que $u_*$ pr\'{e}serve les limites
finies. Donc, il reste \`{a} montrer que $u_*$ pr\'{e}serve les colimites finies. 

\medskip
Nous consid\'{e}rons maintenant un syst\`{e}me inductif $\{\mathcal F_i\}_{i\in I}$ dans
$QCoh(U,u^*(\alpha))$. Choisissons un recouvrement $\mathcal V=\{V_j\}_{j\in J}$ de $X$ tel que 
$\alpha\in \check{C}^2(X,\mathcal O_X^*,\mathcal V)$. Alors, en posant $\mathcal U=\{U_j:=U\times_XV_j\}_{j\in J}$, on a 
$u^*(\alpha)\in \check{C}^2(U,\mathcal O_U^*,\mathcal U)$. Alors, pour chaque $i\in I$, on peut \'{e}crire
$\mathcal F_i$ comme un objet de $QCoh(U,u^*(\alpha),\mathcal U)$: \begin{equation}
\mathcal F_i=(\{\mathcal F_{ij}\}_{j\in J},\{\varphi_{ijj'}\}_{j,j'\in J})
\end{equation} Posons $\mathcal F:=colim_{i\in I}\mathcal F_i$. Alors, on peut pr\'{e}senter
$\mathcal F$ comme un couple $\mathcal F=(\{\mathcal F_j\}_{j\in J},\{\varphi_{jj'}\}_{j,j'\in J})$ o\`{u} $\mathcal F_j\in QCoh(U\times_XV_j)$ est le faisceau associ\'{e} au pr\'{e}faisceau
\begin{equation}\label{2.7}
W\mapsto colim_{i\in I}\mathcal F_{ij}(W) \qquad \forall\textrm{ }W\in ZarAff(U\times_XV_j)
\end{equation} De plus, on sait que $\mathcal F_{ij}$ est un faisceau quasi-coh\'{e}rent 
sur $U\times_XV_j$ pour chaque $i\in I$. \'{E}tant donn\'{e} un morphism $W'\longrightarrow W$
dans $ZarAff(U\times_XV_j)$, on a:
\begin{equation}\label{2.8}
colim_{i\in I}\mathcal F_{ij}(W')=colim_{i\in I}(\mathcal F_{ij}(W)\otimes_{\mathcal O_X(W)}
\mathcal O_X(W'))=(colim_{i\in I}(\mathcal F_{ij}(W)))\otimes_{\mathcal O_X(W)}
\mathcal O_X(W')
\end{equation} Il r\'{e}sulte de \eqref{2.8} que le pr\'{e}faisceau dans \eqref{2.7} est un faisceau quasi-coh\'{e}rent sur $U_j=U\times_XV_j$. Autrement dit, on a 
\begin{equation}
\mathcal F_j(W)=colim_{i\in I}\mathcal F_{ij}(W)\qquad \forall\textrm{ }W\in ZarAff(U\times_XV_j)
\end{equation} De m\^{e}me, pour chaque $W''\in ZarAff(V_j)$, on a 
\begin{equation}\label{2.10}
(colim_{i\in I}u_{j*}\mathcal F_{ij})(W'')=colim_{i\in I}((u_{j*}\mathcal F_{ij})(W''))
\end{equation} o\`{u} $u_j:U_j=U\times_XV_j\longrightarrow V_j$ est le morphisme induit par
$u:U\longrightarrow X$.  Puisque $X$ est s\'{e}par\'{e} et $U$ est affine, $U\times_XW''=
U_j\times_{V_j}W''\in ZarAff(U\times_XV_j)$. Alors, on a 
\begin{equation}\label{2.11}
colim_{i\in I}((u_{j*}\mathcal F_{ij})(W''))=colim_{i\in I}\mathcal F_{ij}(U_j\times_{V_j}W'')=\mathcal F_j(U_j\times_{V_j}W'')=(u_{j*}\mathcal F_j)(W'')
\end{equation} pour chaque $W\in ZarAff(V_j)$. En combinant \eqref{2.10} et \eqref{2.11}, on a un isomorphisme naturel
\begin{equation}\label{2.12}
(colim_{i\in I}u_{j*}\mathcal F_{ij})\cong u_{j*}\mathcal F_j \in QCoh(V_j)
\end{equation} Pour $j,j'\in J$, notons par $u_{jj'}:U_j\cap U_{j'}\longrightarrow V_j\cap V_{j'}$ les
morphismes induits. Puisque $colim_{i\in I}u_*\mathcal F_i\in QCoh(X,\alpha,\mathcal V)$ est le couple 
$(\{(colim_{i\in I}u_{j*}\mathcal F_{ij})\}_{j\in J},\{colim_{i\in I}(u_{jj'*}\varphi_{ijj'})\}_{j,j'\in J})$ et
$u_*\mathcal F\in QCoh(X,\alpha,\mathcal V)$ est le couple $(\{u_{j*}\mathcal F_j\}_{j\in J},\{u_{jj'*}
\varphi_{jj'}\}_{j,j'\in J})$, il s'ensuit que 
\begin{equation}
colim_{i\in I}u_*\mathcal F_i= u_*\mathcal F=u_*(colim_{i\in I}\mathcal F_i)
\end{equation}

\end{proof}

\medskip
Par hypoth\`{e}se, $X$ est un sch\'{e}ma quasi-compact et s\'{e}par\'{e}. Consid\'{e}rons une immersion ouverte $u:U\longrightarrow X$ avec $U$ affine. Alors, $U$ est \'{e}galement un sch\'{e}ma
quasi-compact et s\'{e}par\'{e}. Soit $\alpha\in \check{C}^2(X,\mathcal O_X^*)$ un cocycle de
C\v{e}ch. On sait que la cat\'{e}gorie $QCoh(X,\alpha)$ des faisceaux quasi-coh\'{e}rents sur
$X$ tordu par $\alpha$ est une cat\'{e}gorie ab\'{e}lienne de Grothendieck (voir \cite[Proposition 3.2]{Anti}). Pour savoir plus sur les cat\'{e}gories de Grothendieck, voir, par exemple, \cite[$\S$ 8]{KS}.

\medskip Plus g\'{e}n\'{e}ralement, si $\mathcal A$ est une cat\'{e}gorie ab\'{e}lienne, notons par $K(\mathcal A)$ la cat\'{e}gorie suivante: les objets de $K(\mathcal A)$ sont
des complexes de cha\^{i}nes sur $\mathcal A$ et pour $\mathcal F^\bullet$, $\mathcal G^\bullet\in K(\mathcal A)$, $Hom_{K(\mathcal A)}(\mathcal F^\bullet,\mathcal G^\bullet)$ est l'ensemble des classes d'homotopie des morphismes de $\mathcal F^\bullet$ vers $\mathcal G^\bullet$. Donc, la cat\'{e}gorie d\'{e}riv\'{e}e $D(\mathcal A)$ est obtenue en inversant les quasi-isomorphismes dans $K(\mathcal A)$. Rappelons qu'un objet $\mathcal I^\bullet\in K(A)$  est dit
$K$-injectif si, pour chaque complexe acyclique $\mathcal H^\bullet$ dans $K(\mathcal A)$, 
$Hom^\bullet_{K(\mathcal A)}(\mathcal H^\bullet,\mathcal I^\bullet)$ est acyclique (voir \cite[1.5]{Sp}). En particulier,
$QCoh(X,\alpha)$ \'{e}tant une cat\'{e}gorie de Grothendieck, chaque objet $\mathcal F$ de
$QCoh(X,\alpha)$ dispose d'une r\'{e}solution $K$-injective (voir \cite[Theorem 5.4]{Tarrio}). De m\^{e}me, chaque objet de $QCoh(U,u^*(\alpha))$ dispose d'une r\'{e}solution $K$-injective. Il s'ensuit qu'on a un foncteur d\'{e}riv\'{e}:
\begin{equation}\label{2.14}
R_{qc}u_*:D(QCoh(U,u^*(\alpha)))\longrightarrow D(QCoh(X,\alpha))
\end{equation} De plus, en cons\'{e}quence du Lemme \ref{lem2}, on sait que $u_*:QCoh(U,u^*(\alpha))\longrightarrow QCoh(X,\alpha)$ est un foncteur exact. Alors, on peut
remplacer \eqref{2.14} par un foncteur:
\begin{equation}
u_*:D(QCoh(U,u^*(\alpha)))\longrightarrow D(QCoh(X,\alpha))
\end{equation}

\medskip Nous \'{e}tablissons maintenant le r\'{e}sultat principal de cet article. Soit $\Delta(X,\alpha)$ la
plus petite sous-cat\'{e}gorie triangul\'{e}e de $D(QCoh(X,\alpha))$ contenant tous les objets
$u_*\mathcal F^\bullet$, o\`{u} $u:U\longrightarrow X$ est une immersion ouverte avec $U$ affine et
$\mathcal F^\bullet\in D(QCoh(U,u^*(\alpha)))$. Nous montrons que $\Delta(X,\alpha)=D(QCoh(X,\alpha))$.

\medskip
\begin{thm} Soit $X$ un sch\'{e}ma quasi-compact et s\'{e}par\'{e} et soit $\alpha\in \check{C}^2(X,
\mathcal O_X^*)$ un cocycle de C\v{e}ch. Soit $\Delta(X,\alpha)$ la
plus petite sous-cat\'{e}gorie triangul\'{e}e de $D(QCoh(X,\alpha))$ contenant tous les objets
$u_*\mathcal F^\bullet$, o\`{u} $u:U\longrightarrow X$ est une immersion ouverte avec $U$ affine et
$\mathcal F^\bullet\in D(QCoh(U,u^*(\alpha)))$. Alors,  $\Delta(X,\alpha)=D(QCoh(X,\alpha))$. 
\end{thm}

\begin{proof} Soit $\{U_i\}_{i\in \{1,2,...,n\}}$ un recouvrement affine de $X$. De plus, nous supposons 
que $n=n(X)$, le nombre minimal des affines \`{a} s\'{e}lectionner pour former un recouvrement 
de $X$. On proc\`{e}de par induction sur $n$. Quand $n=1$, le r\'{e}sultat est \'{e}vident. Supposons 
que ce r\'{e}sultat est vrai pour chaque sch\'{e}ma (quasi-compact et s\'{e}par\'{e}) $Z$ tel que 
$n(Z)<n(X)$. En particulier, prenons $X':=\cup_{i=1}^{n-1}U_i$. Alors, $n(X')<n(X)$. On a un morphisme naturel $p:X'\longrightarrow X$ et donc un foncteur d\'{e}riv\'{e}:
\begin{equation}
R_{qc}p_*:D(QCoh(X',p^*(\alpha)))\longrightarrow D(QCoh(X,\alpha))
\end{equation} Choisissons $\mathcal F^\bullet\in D(QCoh(X,\alpha))$. Il est clair que $(p^*,p_*)$ est un couple des foncteurs adjoints. De plus, puisque $p^*:QCoh(X,\alpha)\longrightarrow 
QCoh(X',p^*(\alpha))$ est exact, il s'ensuit que $(p^*,R_{qc}p_*)$ est un couple des foncteurs adjoints entre les cat\'{e}gories d\'{e}riv\'{e}es $D(QCoh(X,\alpha))$ et $D(QCoh(X',p^*(\alpha))$. Alors, on a un morphisme naturel $\mathcal F^\bullet\longrightarrow R_{qc}p_*(p^*\mathcal F)=
R_{qc}p_*(\mathcal F^\bullet|X')$ et donc un triangle distingu\'{e} dans $D(QCoh(X,\alpha))$:
\begin{equation}\label{2.17}
\mathcal G^\bullet \longrightarrow \mathcal F^\bullet \longrightarrow R_{qc}p_*(\mathcal F^\bullet |X')
\longrightarrow \mathcal G^\bullet [1]
\end{equation}

Puisque $n(X')<n(X)$, on a $\Delta(X',p^*(\alpha))=D(QCoh(X',p^*(\alpha)))$. Donc, $\mathcal F^\bullet |X'\in D(QCoh(X',p^*(\alpha)))$ est contenu dans la plus petite sous-cat\'{e}gorie triangul\'{e}e contenant tous les objets $v_*\mathcal H^\bullet$, o\`{u} $v:V\longrightarrow X'$ est une immersion
ouverte avec $V$ affine et $\mathcal H^\bullet\in D(QCoh(V,v^*p^*(\alpha)))$. De plus, si $v:V\longrightarrow X'$ est une immersion ouverte avec $V$ affine,  on voit que :
\begin{equation}\label{2.18}
(R_{qc}p_*)\circ v_*=R_{qc}p_*\circ R_{qc}v_*=R_{qc}(p\circ v)_*=(p\circ v)_* 
\end{equation} (puisque $v_*$ et $(p\circ v)_*$ sont des foncteurs exacts comme une 
cons\'{e}quence du Lemme \ref{lem2}). Il r\'{e}sulte de \eqref{2.18} que 
l'image du foncteur d\'{e}riv\'{e} $R_{qc}p_*:D(QCoh(X',p^*(\alpha)))=\Delta(X',p^*(\alpha))\longrightarrow 
D(QCoh(X,\alpha))$ est contenue dans $\Delta(X,\alpha)$. Donc, $R_{qc}p_*(\mathcal F^\bullet |X')
\in \Delta(X,\alpha)$. 

\medskip
Pour chaque $1\leq i\leq n$, nous notons par $u_i:U_i\longrightarrow X$ l'immersion ouverte
correspondant \`{a} $U_i\in ZarAff(X)$. On va \'{e}tablir que le morphisme naturel
$c:\mathcal G^\bullet\longrightarrow u_{n*}(\mathcal G^\bullet |U_n)$ est un isomorphisme en montrant que $c|U_i:\mathcal G^\bullet |U_i \longrightarrow (u_{n*}(\mathcal G^\bullet |U_n))|U_i$ est un isomorphisme pour chaque $1\leq i\leq n$. 

\medskip
Consid\'{e}rons d'abord le cas o\`{u} $i\ne n$. Puisque
$X'=\cup_{i=1}^{n-1}U_i$, on a un isomorphisme $\mathcal F^\bullet |U_i\overset{\cong}{\longrightarrow} 
R_{qc}p_*(\mathcal F^\bullet |X')|U_i$  et donc $\mathcal G^\bullet |U_i=0$. De plus, on a un carr\'{e} cart\'{e}sien:
\begin{equation}
\begin{CD}
U_n\times_XU_i @>u'_n>> U_i \\
@Vu'_iVV @Vu_iVV \\
U_n @>u_n>> X \\
\end{CD}
\end{equation} Appliquant Lemme \ref{lem1}, on a 
\begin{equation}
(u_{n*}(\mathcal G^\bullet |U_n))|U_i= u_i^*u_{n*}(\mathcal G^\bullet |U_n)\cong u'_{n*}u'^*_i(\mathcal G^\bullet |U_n)=u'_{n*}u'^*_iu_n^*\mathcal G^\bullet = u'_{n*}u'^*_n(\mathcal G^\bullet |U_i)=0
\end{equation} et donc un isomorphisme $c|U_i:\mathcal G^\bullet |U_i =0 \longrightarrow (u_{n*}(\mathcal G^\bullet |U_n))|U_i=0$ pour chaque $i\ne n$. 

\medskip
Il reste \`{a} montrer le cas o\`{u} $i=n$. Dans ce cas, il est clair que $c|U_n:\mathcal G^\bullet |U_n\longrightarrow (u_{n*}(\mathcal G^\bullet |U_n))|U_n=u_n^*u_{n*}(\mathcal G^\bullet |U_n)$ 
est un isomorphisme. Par suite, on a un isomorphisme:
\begin{equation}
c:\mathcal G^\bullet\overset{\cong}{\longrightarrow} u_{n*}(\mathcal G^\bullet |U_n)
\end{equation} et donc $\mathcal G^\bullet\in \Delta(X,\alpha)$. Puisque 
$R_{qc}p_*(\mathcal F^\bullet |X')$ est \' {e}galement dans $\Delta(X,\alpha)$, il r\'{e}sulte
du triangle \eqref{2.17} que $\mathcal F^\bullet \in \Delta(X,\alpha)$. Donc, $\Delta(X,\alpha)
=D(QCoh(X,\alpha))$. 

\end{proof}

\end{document}